\documentclass[preprint,12pt]{elsarticle1}

\usepackage{subfig}
\usepackage{graphicx}
\usepackage{caption,color}   
\usepackage{amssymb}
\usepackage{amsthm}
\usepackage{amsmath}
\usepackage{epic}
\usepackage{setspace}
\usepackage{float}
\usepackage{multirow}

\usepackage{hyperref}
\usepackage{xcolor}
\usepackage{marginnote}

\usepackage{amsthm}

\newtheorem{thm}{Theorem}[section]
\newtheorem{cor}[thm]{Corollary}

\newtheorem{prou}[thm]{Definition}
\newtheorem{rmk}[thm]{Remark}
\newtheorem{example}[thm]{Example}

\numberwithin{equation}{section}
\journal{}

\begin{document}
\begin{spacing}{1.15}
\begin{frontmatter}
\title{\textbf{ Quadratic form estimations for Hessian matrices of resistance distance and Kirchhoff index of positive-weighted graphs }}

\author[label1]{Yu Li} \ead{liyuzone@163.com}
\author[label1]{Lizhu Sun\corref{cor}}\ead{lizhusun@hrbeu.edu.cn}
\author[label1]{Changjiang Bu} \ead{buchangjiang@hrbeu.edu.cn}

\cortext[cor]{Corresponding author}

\address[label1]{College of Mathematical Sciences, Harbin Engineering University, Harbin, PR China}

\begin{abstract}

Let $G^{w}=(V,E,w)$ be a positive-weighted graph with the weight $w(e)>0$ for all $e\in E.$
 The weighted graph $G^{\widetilde{w}}=(V,E,\widetilde{w})$  is called a hyper-dual number weighted graph, where the  weight  $\widetilde{w}(e)=w(e)+\Delta w(e)(\varepsilon+\varepsilon^{*})$ is a hyper dual number, $\Delta w(e)$ is a real number,  $\varepsilon$ and $\varepsilon^{*}$ are two dual units, $e\in E$. In this paper, we  give a representation for the Moore-Penrose inverse of the Laplacian matrix, and calculation formulas for the resistance distance and Kirchhoff index of  $G^{\widetilde{w}}$, respectively.   We establish  quadratic forms of the Hessian matrices for the resistance distance and Kirchhoff index of
$G^{w}$ via generalized matrix  inverses. We further derive explicit bounds on the eigenvalues of the Hessian matrices for the resistance distance and the Kirchhoff index of
$G^{w}$  in terms of  graph parameters. We  also prove that the Kirchhoff index of a positive-weighted graph with bounded edge weights is strongly convex on its edge weight vector.

\end{abstract}

\begin{keyword} Hessian matrix; hyper-dual number weighted  graph; resistance distance; Kirchhoff index;  generalized inverse \\
\emph{AMS classification(2020):} 05C09; 05C12; 05C50; 15A10.
\end{keyword}

\end{frontmatter}

\section{Introduction}

All graphs considered in this paper are simple undirected \cite{harary1969graph}. For a connected graph $G$, when a resistance
 of 1 ohm is placed on each edge, the effective resistance between vertices $i$ and $j$ of $G$
is called the resistance distance between $i$ and $j$, denoted by $R_{ij}(G)$ \cite{foster1949average}. The sum of the resistance distances over all pairs of vertices is called the Kirchhoff index of $G$, denoted by $Kf(G)$ \cite{chen2008resistance}. In 1993, Klein and Randi\'{c} gave formulas for the resistance distance and Kirchhoff index of a connected graph $G$ by using the Moore-Penrose inverse of its Laplacian matrix \cite{Klein1993Resistance}.
Let $G^{w}$ be a connected positive-weighted graph, where $w(e)>0$ is the weight of its each edge $e$. 
In 2004, Bapat gave formulas for the resistance distance and Kirchhoff index of $G^{w}$ \cite{bapat2004resistance}. The resistance distance and Kirchhoff index have been applied in various research, including graph spanning trees \cite{cheng2022counting,li2019enumeration}, random walks \cite{lovasz1993random,cloninger2024random}, network analysis \cite{zhang2019detecting}, chemical graph theory \cite{bonchev1994molecular}, and robust network design \cite{yang2018designing,ellens2011effective}.

Let $\textbf{y}=(y_{1},y_{2},\cdots,y_{m})^{\top}$, where  $y_{i}$ and $\frac{1}{y_{i}}$ are the resistance  and   weight of the edge  $e_{i}$ in a positive-weighted graph $G^{w}$, $i=1,2,\cdots,m.$ In 1956, Shannon and Hagelberger proved that the resistance distance of $G^{w}$ is a concave function of the resistance vector $\textbf{y}$ \cite{shannon1993concavity}. In 2008, Ghosh, Boyd, and Saberi showed that the resistance distance of $G^{w}$ is a convex function of the edge weight vector $\textbf{x}=(\frac{1}{y_{1}},\frac{1}{y_{2}},\cdots,\frac{1}{y_{m}})^{\top}$ \cite{ghosh2008minimizing}, but this conclusion does not follow directly from the result in \cite{shannon1993concavity}. Furthermore, Ref. \cite{ghosh2008minimizing} gave the quadratic forms of the Hessian matrix for the Kirchhoff index of $G^{w}$ via the  matrix inverses, and proved that the Kirchhoff index of $G^{w}$ is a strictly convex function of $\textbf{x}$.  By using dual numbers, Ref. \cite{li2025resistance} gave the error estimates for the resistance distance and Kirchhoff index of a graph.

 Let $\mathbb{R}^{m}$ denote the set of $m$-dimensional vectors over the real number field $\mathbb{R}$ and $\mathbf{x}=(x_{1},x_{2},\cdots,x_{m})^{\top}\in \mathbb{R}^{m}.$ Let $f(\textbf{x})$ be a real-valued function on $\mathbb{R}^{m}$ with continuous second order partial derivatives.
The matrix $$\nabla^{2} f(\textbf{x})=\left(\frac{\partial^{2} f}{\partial x_{i}\partial x_{j}}\right)_{m \times m}$$ is called the Hessian matrix of
$f(\mathbf{x})$. The quadratic form and eigenvalues of a Hessian matrix of $f(\mathbf{x})$ find broad applications in neural network optimization \cite{chen2021fast}, algebraic graph theory \cite{yazawa2021eigenvalues} and principal component analysis of graphs \cite{pan2019principal}.

The hyper-dual numbers are a powerful tool for obtaining the quadratic form of the Hessian matrix of a function. In this paper, we give the Moore-Penrose inverse of the Laplacian matrix for a connected hyper-dual number weighted graph, and derive explicit formulas for its resistance distance and Kirchhoff index.
By using these results, we give the quadratic forms of the Hessian matrices for the resistance distance and Kirchhoff index of a positive-weighted graph $G^{w}$ via  the Moore-Penrose inverse of the Laplacian matrix of $G^{w}$. We also establish explicit bounds on these quadratic forms, which are expressed  in terms of the largest eigenvalue of the Laplacian matrix, algebraic connectivity,  biharmonic distance of $G^{w}$, and maximum degree of the underlying  graph $G$. We further prove that the Kirchhoff index of a positive-weighted graph
$G^{w}$  with bounded edge weights  is strongly convex on its edge weight vector.

\section{Preliminaries}
\subsection{Hyper-dual numbers}
In 1873, Clifford introduced the dual number \cite{clifford1871preliminary}
$$\widehat{a}=a+a_{1}\varepsilon,$$
where $\varepsilon$ is the dual unit satisfying $\varepsilon\neq0, \varepsilon^{2}=0,$ and $a,a_{1}$ are real numbers. Up to now, dual numbers have been widely applied in mechanics and kinematics \cite{Geometry1010,fischer}, robotics \cite{gu1987dual}, formation control \cite{qi2024eigenvalues}, brain science \cite{wei2024singular} and Markov chains \cite{qi2024dualww} and so on.
In 2011, Fike and Alonso proposed the hyper-dual number \cite{fike2011development}
$$\widetilde{a}=a+a_{1}\varepsilon+a_{2}\varepsilon^{*}
 +a_{3}\varepsilon\varepsilon^{*},$$
 where $\varepsilon,\varepsilon^{*}$ are two dual units satisfying $\varepsilon\varepsilon^{*}=\varepsilon^{*}\varepsilon\neq 0, $ $ \varepsilon^{2}=(\varepsilon^{*})^{2}
=0,$ and $a,a_{1},a_{2},a_{3}$ are real numbers. Hyper-dual numbers are widely utilized in automatic differentiation, multibody kinematics and optimization problems \cite{fike2012automatic,murai2022method,cohen2016application}.

Let $\mathbb{H}^{m}$ be the set of
$m$-dimensional vectors over the hyper-dual number set $\mathbb{H}$. Let $\widetilde{\textbf{x}}=\textbf{x}+\textbf{x}_{1}\varepsilon+\textbf{x}_{2}\varepsilon^{*}+\textbf{x}_{3}\varepsilon\varepsilon^{*}\in \mathbb{H}^{m},$ where $\textbf{x},\textbf{x}_{1},\textbf{x}_{2},\textbf{x}_{3}\in\mathbb{R}^{m}$.
Let $f(\widetilde{\textbf{x}})=f(\textbf{x}+\textbf{x}_{1}\varepsilon+ \textbf{x}_{2}\varepsilon^{*}+ \textbf{x}_{3}\varepsilon\varepsilon^{*})$ be the hyper-dual function of
$\widetilde{\textbf{x}}$. Let
 $\mathcal{S}_{\varepsilon\varepsilon^{*}}[f(\widetilde{\mathbf{x}})]$
denote the real coefficient of $\varepsilon\varepsilon^{*}$ in $f(\widetilde{\mathbf{x}})$.
By using hyper-dual numbers, Refs. \cite{fike2011development,tanaka2015highly} gave a formula for calculating  the quadratic form  $(\Delta \textbf{x})^{\top}\nabla^{2}f(\textbf{x})\Delta \textbf{x}$ of the Hessian matrix of a  function $f(\textbf{x})$ given by
\begin{align}
\label{jiq2}
(\Delta \textbf{x})^{\top}\nabla^{2}f(\textbf{x})\Delta \textbf{x}&=\mathcal{S}_{\varepsilon\varepsilon^{*}}[f(\textbf{x} +\Delta \textbf{x}(\varepsilon+\varepsilon^{*}))],
\end{align}
where $\Delta \textbf{x}\in\mathbb{R}^{m}.$
\subsection{Hyper-dual number weighted graphs}
A hyper-dual number matrix is a matrix whose entries are hyper-dual numbers. Let $\mathbb{R}^{m\times n}$ and $\mathbb{H}^{m\times n}$ denote the sets of all real matrices and hyper-dual number matrices of order $m\times n$, respectively.
Write $[n]=\{1,2,\cdots,n\}$. Next, we define the hyper-dual  number weighted graph.

\begin{prou}
Let $G=(V,E)$ be a graph with the vertex set $V$ and edge set
$E=\{e_{1},e_{2},\cdots,e_{m}\}$. Let the hyper-dual number weight of  edge $e_{i}$ be $$\widetilde{w}(e_{i})=\widetilde{x}_{i}=x_{i}+\Delta x_{i}(\varepsilon+\varepsilon^{*}),
$$ where  $x_{i}=w(e_{i})>0$ is the weight of edge $e_{i}$ in the  positive-weighted graph $G^{w}=(V,E,w)$, and  $\Delta x_{i}= \Delta w(e_{i})$ is the real number satisfying  $x_{i}+\Delta x_{i}>0$, for all $i\in[m]$.  We call $G^{\widetilde{w}}=(V,E,\widetilde{w})$ the hyper-dual number weighted graph.
\end{prou}
The vector  $$\widetilde{\textbf{x}}=(\widetilde{x}_{1},\widetilde{x}_{2},
\cdots,\widetilde{x}_{m})^{\top}=\textbf{x}+\Delta\textbf{x}
(\varepsilon+\varepsilon^{*})$$ is called the hyper-dual edge weight vector of
$G^{\widetilde{w}}$, where $\mathbf{x}=(x_{1},x_{2},\cdots,x_{m})^{\top}$ is the edge weight vector of
$G^{w}$, $\Delta \mathbf{x}=(\Delta x_{1}, \Delta x_{2},\cdots, \Delta x_{m})^{\top}$. For a connected  positive-weighted graph $G^{w}$,  its resistance distance $R_{ij}(G^{w})$ between vertices  $i$ and $j$ and  Kirchhoff index $Kf(G^{w})$ are denoted by  $R_{ij}(\textbf{x})$ and $Kf_{G}(\textbf{x})$, respectively.

Let  $A=(a_{ij})$ be  the adjacency matrix of  $G^{w}$, where $$a_{ij}
= \begin{cases}
w(e), & e=\{i,j\}\in E, \\
0, & \textrm{otherwise}.
\end{cases}$$
We call $\widetilde{A}=(\widetilde{a}_{ij})
=A+A_{1}(\varepsilon+\varepsilon^{*})$
 the adjacency matrix of $G^{\widetilde{w}}$, where $A_{1}=(\Delta a_{ij})\in \mathbb{R}^{n\times n}$. Then
\begin{equation}\label{chaoduiou}
    \begin{aligned}
\widetilde{a}_{ij}
&= \begin{cases}
a_{ij}+
\Delta a_{ij}(\varepsilon+\varepsilon^{*}), & e=\{i,j\}\in E, \\
0, & \textrm{otherwise}
\end{cases}
\\&= \begin{cases}
w(e)+
\Delta w(e)(\varepsilon+\varepsilon^{*}), & e=\{i,j\}\in E, \\
0, & \textrm{otherwise}.
\end{cases}
 \end{aligned}
\end{equation}

The hyper-dual number matrix \begin{align}\label{D1L} \widetilde{D}=D+D_{1}(\varepsilon+\varepsilon^{*})=\textrm{diag}(\widetilde{d}_{1},
\widetilde{d}_{2},\cdots,\widetilde{d}_{n})\end{align} is called  the degree matrix of $G^{\widetilde{w}}$, where $D=\textrm{diag}(d_{1},d_{2},\cdots,d_{n})$ is the degree matrix of  $G^{w}$, $D_{1}=\textrm{diag}(\Delta d_{1},\Delta d_{2}
 ,\cdots,\Delta d_{n})$, $d_{i}$ is the weighted degree of vertex  $i$ in $G^{w}$ and  $\Delta d_{i}=\sum\limits_{j=1}^{n}\Delta a_{ij}$, for all $i\in[n]$.
We call
\begin{equation}
    \begin{aligned}\label{lapu1}
\widetilde{L}=\widetilde{D}-\widetilde{A}
=L+L_{1}(\varepsilon+\varepsilon^{*})
 \end{aligned}
\end{equation}
 the Laplacian matrix of $G^{\widetilde{w}}$, where $L=D-A$  is the Laplacian matrix of  $G^{w}$, $L_{1}=D_{1}-A_{1}\in \mathbb{R}^{n\times n}$.

For a connected hyper-dual number weighted graph $G^{\widetilde{w}}$, the resistance distance $R_{ij}(G^{\widetilde{w}})$ between any two vertices $i,j$ and  Kirchhoff index $Kf(G^{\widetilde{w}})$ of  $G^{\widetilde{w}}$ are denoted by
 functions $R_{ij}(\widetilde{\textbf{x}})$ and
$Kf_{G}(\widetilde{\textbf{x}})$ of the hyper-dual edge weight vector $\widetilde{\textbf{x}}$, respectively.

\subsection{The Moore-Penrose inverse of hyper-dual number matrices}
Let $I$ be the identity matrix of appropriate order. Let $\widetilde{A}\in \mathbb{H}^{n\times n}$.  If there exists $\widetilde{X}\in\mathbb{H}^{n\times n}$ such that $\widetilde{A}\widetilde{X}=\widetilde{X}\widetilde{A}=I$, then $\widetilde{A}$ is invertible and $\widetilde{X}$ is called the inverse of $\widetilde{A}$, denoted by $\widetilde{A}^{-1}$. If there exists  $\widetilde{X}\in\mathbb{H}^{n\times m}$  satisfying
\begin{align}\label{mp2}\widetilde{A}\widetilde{X}\widetilde{A}=\widetilde{A}, \ \ \widetilde{X}\widetilde{A}\widetilde{X}
=\widetilde{X},\ \  (\widetilde{A}\widetilde{X})^{\top}
=\widetilde{A}\widetilde{X} , \ \ (\widetilde{X}\widetilde{A})^{\top}
=\widetilde{X}\widetilde{A}, \end{align}
then $\widetilde{X}$ is called the Moore-Penrose inverse of $\widetilde{A}$, denoted by $\widetilde{A}^{\dag}.$
 If the Moore-Penrose inverse of a hyper-dual number matrix exists, it is unique \cite{xiao2024c}.
In order to derive  the resistance distance and Kirchhoff index of a hyper-dual number weighted graph, we next give the existence and explicit representation of solutions to hyper-dual linear equations.

\begin{thm}\label{3.3}
Let $\widetilde{A}\in\mathbb{H}^{m\times n}$, $ \widetilde{\mathbf{x}}\in\mathbb{H}^{n}$ and $\widetilde{\mathbf{b}}\in\mathbb{H}^{m}$. Suppose that $\widetilde{A}^{\dag}$ exists. Then the following statements hold.

$(a)$ The solution  to the hyper-dual  equation
$\widetilde{A}\widetilde{\mathbf{x}}=\widetilde{\mathbf{b}}$ exists  if and only if $$\widetilde{A}\widetilde{A}^{\dag}
\widetilde{\mathbf{b}}=\widetilde{\mathbf{b}}.$$

$(b)$  If the solution to the hyper-dual equation $\widetilde{A}\widetilde{\mathbf{x}}=\widetilde{\mathbf{b}}$ exists, then its general solution is given by \begin{align}\label{eq91} \widetilde{\mathbf{x}}=\widetilde{A}^{\dag}\widetilde{\mathbf{b}}
+(I-\widetilde{A}^{\dag}\widetilde{A})\widetilde{\mathbf{u}},
\end{align} where $\widetilde{\mathbf{u}}\in \mathbb{H}^{n}$ is arbitrary.
\end{thm}
\begin{proof}
$(a)$ Suppose $\widetilde{\mathbf{x}}$ is a solution to $\widetilde{A}
\widetilde{\mathbf{x}}=\widetilde{\mathbf{b}}$. Then
$$\widetilde{\mathbf{b}}=\widetilde{A}
\widetilde{\mathbf{x}}=(\widetilde{A}\widetilde{A}^{\dag})
\widetilde{A}
\widetilde{\mathbf{x}}=
\widetilde{A}\widetilde{A}^{\dag}
\widetilde{\mathbf{b}}.$$ Conversely, if
$\widetilde{A}\widetilde{A}^{\dag}
\widetilde{\mathbf{b}}
=\widetilde{\mathbf{b}}$, it is clear that $\widetilde{A}^{\dag}\widetilde{\mathbf{b}}$ is a solution to $\widetilde{A}
\widetilde{\mathbf{x}}
=\widetilde{\mathbf{b}}$.

 $(b)$  From $(a)$, we have $\widetilde{A}\widetilde{A}^{\dag}
 \widetilde{\mathbf{b}}=\widetilde{\mathbf{b}}$. For any $\widetilde{\mathbf{u}}\in \mathbb{H}^{m}$,   $$
\widetilde{A}(\widetilde{A}^{\dag}
\widetilde{\mathbf{b}}+
(I-\widetilde{A}^{\dag}\widetilde{A})
\widetilde{\mathbf{u}})
=\widetilde{A}\widetilde{A}^{\dag}\widetilde{\mathbf{b}}
=\widetilde{\mathbf{b}}.$$ Hence, $\widetilde{\mathbf{x}}=\widetilde{A}^{\dag}\widetilde{\mathbf{b}}
+(I-\widetilde{A}^{\dag}\widetilde{A})
\widetilde{\mathbf{u}}$ is a solution to $\widetilde{A}
\widetilde{\mathbf{x}}=\widetilde{\mathbf{b}}$.

On the other hand, let
$\widetilde{\mathbf{x}}$ be a solution to $\widetilde{A}
\widetilde{\mathbf{x}}=\widetilde{\mathbf{b}}$.
 Then
\begin{align}\nonumber \widetilde{\mathbf{x}}&=\widetilde{\mathbf{x}}
+\widetilde{A}^{\dag}(\widetilde{\mathbf{b}}-
\widetilde{A}\widetilde{\mathbf{x}})\\ \nonumber
&=\widetilde{A}^{\dag}\widetilde{\mathbf{b}}+
(I-\widetilde{A}^{\dag}\widetilde{A})\widetilde{\mathbf{x}}.
\end{align}
Therefore, every solution to $\widetilde{A}\widetilde{\mathbf{x}}
=\widetilde{\mathbf{b}}$ can be written in the form Eq. \eqref{eq91}.
\end{proof}

\section{Main result}
\subsection{Resistance distance  and  Kirchhoff index of hyper-dual number  weighted graphs}
In this paper, we use $\textbf{1}$,
$\textbf{1}_{i}$ and $J$ to denote
the all-ones vector, the standard unit vector  of appropriate dimension with  1 at the $i$-th position and 0 elsewhere, and the all-ones square matrix of appropriate order, respectively.
For a matrix $P\in \mathbb{H}^{n\times n}$, we write $P^{(i,j)}=(P)_{ii}+(P)_{jj}-(P)_{ij}-(P)_{ji}.$

This section establishes a representation  for the Moore-Penrose inverse of the Laplacian matrix of a connected hyper-dual number weighted graph $G^{\widetilde{w}}$, and then  gives explicit  calculation formulas
 for the resistance distance and Kirchhoff index of $G^{\widetilde{w}}$ by means of this generalized inverse.
\begin{thm}\label{5.1}Let $\widetilde{L}=
L+L_{1}(\varepsilon+\varepsilon^{*})$ given in Eq. \eqref{lapu1} be the Laplacian matrix of a connected hyper-dual number weighted graph $G^{\widetilde{w}}$  with
$n$ vertices. Then

$(a)$ The Moore-Penrose inverse of $\widetilde{L}$ exists and \begin{align}\label{eq1}
\widetilde{L}^{\dag}=L^{\dag}-
L^{\dag}L_{1}L^{\dag}(\varepsilon+\varepsilon^{*})
+2L^{\dag}L_{1}L^{\dag}L_{1}L^{\dag}
\varepsilon\varepsilon^{*}.
\end{align}

$(b)$
 \begin{align}\nonumber
R_{ij}(G^{\widetilde{w}})
=&(L^{\dag})^{(i,j)}-
(L^{\dag}L_{1}L^{\dag})^{(i,j)}(\varepsilon
+\varepsilon^{*})
+2(L^{\dag}L_{1}L^{\dag}L_{1}L^{\dag})^{(i,j)}\varepsilon\varepsilon^{*},
\end{align}
where $i,j\in[n]$.

$(c)$
\begin{align}\label{jia}
Kf(G^{\widetilde{w}})
=n\textrm{tr}(L^{\dag})
-n\textrm{tr}((L^{\dag})^{2}L_{1})
(\varepsilon+\varepsilon^{*})
+2n\textrm{tr}((L^{\dag})^{2}L_{1}L^{\dag}L_{1}\varepsilon\varepsilon^{*}.
\end{align}
\end{thm}

\begin{proof}
$(a)$ Let $\widetilde{X}=L^{\dag}-
L^{\dag}L_{1}L^{\dag}(\varepsilon+\varepsilon^{*})
+2L^{\dag}L_{1}L^{\dag}L_{1}L^{\dag}
\varepsilon\varepsilon^{*}.$  We now prove that $\widetilde{X}$ is the Moore-Penrose inverse of $\widehat{L}$.
It follows that \begin{align}\nonumber
\widetilde{X}\widetilde{L}&
=(L^{\dag}-
L^{\dag}L_{1}L^{\dag}(\varepsilon+\varepsilon^{*})
+2L^{\dag}L_{1}L^{\dag}L_{1}L^{\dag}
\varepsilon\varepsilon^{*})
(L+L_{1}(\varepsilon+\varepsilon^{*}))
\\ \nonumber
&=L^{\dag}L-
L^{\dag}L_{1}L^{\dag}L(\varepsilon+\varepsilon^{*})+
2L^{\dag}L_{1}L^{\dag}L_{1}L^{\dag}L
\varepsilon\varepsilon^{*}+L^{\dag}L_{1}(\varepsilon+\varepsilon^{*})-
2L^{\dag}L_{1}L^{\dag}L_{1}\varepsilon\varepsilon^{*}\\ \label{mooreP}
&=L^{\dag}L+L^{\dag}L_{1}(I-L^{\dag}L)(\varepsilon+\varepsilon^{*})
-2L^{\dag}L_{1}L^{\dag}L_{1}(I-L^{\dag}L)\varepsilon\varepsilon^{*}.
\end{align}
 Since the  positive-weighted graph
$G^{w}$ is connected, its Laplacian matrix $L$ is positive semidefinite with rank $n-1$\cite[p.~50]{bapat2010graphs}.
 Let $\lambda_{1}\geq\lambda_{2}
 \geq\cdots\geq\lambda_{n-1}>\lambda_{n}=0$  be the eigenvalues of $L$, and let $p_{i}$ be the unit eigenvector corresponding to $\lambda_{i}$ such that $p_{i}$ and $p_{j}$ are mutually orthogonal for $i,j\in[n]$, $i\neq j$.
 Since the all-ones vector is an eigenvector of $L$ corresponding to $\lambda_{n}=0$,  we choose $p_{n}=\frac{1}{\sqrt{n}}\mathbf{1}$. Then \begin{align}\label{pfj}
  L=\sum_{i=1}^{n}\lambda_{i}p_{i}p_{i}^{\top}
  =\sum_{i=1}^{n-1}\lambda_{i}p_{i}p_{i}^{\top},\end{align}
  and the Moore-Penrose inverse of $L$ is \begin{align}\label{pfj1}L^{\dag}=
\sum_{i=1}^{n-1}\frac{1}{\lambda_{i}}p_{i}p_{i}^{\top}.\end{align}
From Eqs. \eqref{pfj} and \eqref{pfj1}, we obtian
   \begin{align}\nonumber I-L^{\dag}L&=I-\sum_{i=1}^{n-1}\frac{1}{\lambda_{i}}p_{i}p_{i}^{\top}
   \sum_{i=1}^{n-1}
   \lambda_{i}p_{i}p_{i}^{\top}=
I-\sum_{i=1}^{n-1}p_{i}p_{i}^{\top}\\ \nonumber
&=\sum_{i=1}^{n}p_{i}p_{i}^{\top}-
\sum_{i=1}^{n-1}p_{i}p_{i}^{\top}=p_{n}p_{n}^{\top}=
\frac{1}{\sqrt{n}\sqrt{n}}\textbf{1}\textbf{1}^{\top}=\frac{1}{n}
J.\end{align}
Substituting the above equality  into Eq. \eqref{mooreP} gives $
\widetilde{X}\widetilde{L}=I-\frac{1}{n}
J+\frac{1}{n}L^{\dag}L_{1}J(\varepsilon+\varepsilon^{*})
-\frac{2}{n}L^{\dag}L_{1}
L^{\dag}L_{1}J\varepsilon\varepsilon^{*}.$ Since  the row sums of $L_{1}$ are zero,  it follows that\begin{align}\label{npd}
\widetilde{X}\widetilde{L}=I-\frac{1}{n}J.
\end{align}
Similarly, $\widetilde{L}\widetilde{X}=I-\frac{1}{n}J.$
Since both $\widetilde{L}\widetilde{X}$ and $\widetilde{X}
\widetilde{L}$
  are symmetric,  $\widetilde{X}$ satisfies the third and fourth equations in Eq. \eqref{mp2}.

  Multiplying both sides of Eq. \eqref{npd} from the left by $\widetilde{L}$  gives $$\widetilde{L}\widetilde{X}\widetilde{L}=\widetilde{L}(I-\frac{1}{n}J)=\widetilde{L}-
\frac{1}{n}\widetilde{L}J=\widetilde{L}.$$
Thus, $\widetilde{X}$  satisfies the first equation in Eq. \eqref{mp2}.

Multiplying both sides of Eq.~\eqref{npd} from the left by $\widetilde{X}$ yields
\begin{align}\label{xlx}\widetilde{X}\widetilde{L}\widetilde{X}=(I-\frac{1}{n}J)\widetilde{X}=\widetilde{X}-
\frac{1}{n}J\widetilde{X}.\end{align}
Since $L$ is symmetric, it follows that $LL^{\dag}
=(LL^{\dag})^{\top}
=(L^{\dag})^{\top}L^{\top}
=L^{\dag}L.$
Then $JL^{\dag}=
JL^{\dag}LL^{\dag}
=JL(L^{\dag})^{2}=0.$ Hence, $J
\widetilde{X}=JL^{\dag}-
JL^{\dag}L_{1}L^{\dag}(\varepsilon+\varepsilon^{*})
+2JL^{\dag}L_{1}L^{\dag}L_{1}L^{\dag}
\varepsilon\varepsilon^{*}=0$.
Again by Eq. \eqref{xlx}, we have
$\widetilde{X}\widetilde{L}\widetilde{X}=
\widetilde{X},$
which  shows that $\widetilde{X}$ satisfies the second equation in Eq. \eqref{mp2}.
Therefore,
$\widetilde{X}$  is the Moore-Penrose inverse of $\widetilde{L}$, i.e.,
$\widetilde{L}^{\dag}=L^{\dag}-
L^{\dag}L_{1}L^{\dag}(\varepsilon+\varepsilon^{*})
+2L^{\dag}L_{1}L^{\dag}L_{1}L^{\dag}
\varepsilon\varepsilon^{*}.$

$(b)$
If $i=j$, clearly $R_{ij}(G^{\widetilde{w}})=0$. Assume  that
 $i\neq j$, $i,j\in[n]$ and  a voltage is applied  between vertices $i$ and $j$ of $G^{\widetilde{w}}$. Let  $Y\in \mathbb{H}$ denote the net current flowing from vertices $i$ to $j$ of $G^{\widetilde{w}}$. Let $\textbf{v}=(v_{1},v_{2},\cdots,v_{n})^{T}\in \mathbb{H}^{n}$ be the electric potential vector, where $v_{s}$ is the potential at the vertex $s$ of $G^{\widetilde{w}}$, for $s\in[n]$. Let $y_{st}$ denote the current flowing from a vertex $s$ to its adjacent vertex $t$ $(y_{st}=-y_{ts})$. Then
\begin{align}\label{lmt}
\sum_{s\sim t}y_{st}=Y(\textbf{1}_{i}-\textbf{1}_{j})_{s}=\begin{cases}
Y, & s=i,\\
-Y, & s=j,\\
0,  & s\neq i,j,
\end{cases}
\end{align}
where $s\sim t$ denote vertices $s$ and $t$ are adjacent.

Let $\Omega_{st}=\frac{1}{\widetilde{a}_{st}}$ be
 the resistance on edge $e=\{s,t\}$ in
 $G^{\widetilde{w}}$.
Then \begin{align}\nonumber
\sum_{s\sim t}y_{st}&=\sum_{s\sim t}\frac{v_{s}-v_{t}}{\Omega_{st}}=\sum_{s\sim t}\widetilde{a}_{st}(v_{s}-v_{t})=v_{s}\sum_{s\sim t}\widetilde{a}_{st}-\sum_{s\sim t}\widetilde{a}_{st}v_{t}\\ \label{qs}
&=\widetilde{d}_{s}v_{s}-(\widetilde{A}\textbf{v})_{s}=((\widetilde{D}-\widetilde{A})\textbf{v})_{s}=(\widetilde{L}\textbf{v})_{s}.
\end{align}
From Eqs. \eqref{lmt} and \eqref{qs}, we have
\begin{align}\label{sh}\widetilde{L}\textbf{v}=Y(\textbf{1}_{i}-\textbf{1}_{j}).\end{align}
By Theorem \ref{3.3} $(b)$, the general solution to Eq. \eqref{sh} is
 $$\textbf{v}=\widetilde{L}^{\dag} Y(\textbf{1}_{i}-\textbf{1}_{j})+(I-\widetilde{L}^{\dag}\widetilde{L})
\textbf{u},$$
where $\textbf{u}\in\mathbb{H}^{n}$ is arbitrary. Hence, the resistance distance between vertices $i$ and $j$ of $G^{\widetilde{w}}$ is\begin{align}\nonumber
R_{ij}(G^{\widetilde{w}})&=
\frac{v_{i}-v_{j}}{Y}=
\frac{1}{Y}(\textbf{1}_{i}-\textbf{1}_{j})^{\top}\textbf{v}\\ \label{bb}
&=\frac{1}{Y}(\textbf{1}_{i}-\textbf{1}_{j})^{\top}
\widetilde{L}^{\dag} Y(\textbf{1}_{i}-\textbf{1}_{j})+
\frac{1}{Y}(\textbf{1}_{i}-\textbf{1}_{j})^{\top}
(I-\widetilde{L}^{\dag}\widetilde{L})
\textbf{u}.\end{align}
From Eq. \eqref{npd}, we have $I-\widetilde{L}^{\dag}\widetilde{L}=\frac{1}{n}J.$
Then   the second term of Eq. \eqref{bb}  is $$\frac{1}{Y}(\textbf{1}_{i}-\textbf{1}_{j})^{\top}
(I-\widetilde{L}^{\dag}\widetilde{L})=
\frac{1}{Y}(\textbf{1}_{i}-\textbf{1}_{j})^{\top}(\frac{1}{n}J)=0.$$
Therefore,
\begin{align}\label{fa2}R_{ij}(G^{\widetilde{w}})
=(\textbf{1}_{i}-\textbf{1}_{j})^{\top}\widetilde{L}^{\dag} (\textbf{1}_{i}-\textbf{1}_{j})=(\widetilde{L}^{\dag})^{(i,j)}.\end{align}
Substituting Eq. \eqref{eq1} into Eq. \eqref{fa2}, we have
\begin{align}\nonumber R_{ij}(G^{\widetilde{w}})
=(L^{\dag})^{(i,j)}-
(L^{\dag}L_{1}L^{\dag})^{(i,j)}(\varepsilon
+\varepsilon^{*})
+2(L^{\dag}L_{1}L^{\dag}L_{1}L^{\dag})^{(i,j)}
\varepsilon\varepsilon^{*}.
\end{align}

$(c)$ From part $(a)$, we know that $\widetilde{L}^{\dag}$ exists.  It follows that
\begin{align}\nonumber
Kf(G^{\widetilde{w}})&=\frac{1}{2}\sum_{i,j=1}^{n}R_{ij}(G^{\widetilde{w}})
=\frac{1}{2}\sum_{i,j=1}^{n}
(\widetilde{L}^{\dag})^{(i,j)}\\ \nonumber
&=\frac{1}{2}\sum_{i,j=1}^{n}(\widetilde{L}^{\dag})_{ii}+
\frac{1}{2}\sum_{i,j=1}^{n}(\widetilde{L}^{\dag})_{jj}
-\frac{1}{2}\sum_{i,j=1}^{n}(\widetilde{L}^{\dag})_{ij}
-\frac{1}{2}\sum_{i,j=1}^{n}(\widetilde{L}^{\dag})_{ji}\\ \nonumber
&=n \textrm{tr}(\widetilde{L}^{\dag})-
\textbf{1}^{T}\widetilde{L}^{\dag}\textbf{1}\\
\label{jm}
&=n \textrm{tr}(\widetilde{L}^{\dag}).
\end{align}
Substituting Eq. \eqref{eq1} into Eq. \eqref{jm} yields
\begin{align}\nonumber Kf(G^{\widetilde{w}})
&=n\textrm{tr}(L^{\dag})-n\textrm{tr}(L^{\dag}L_{1}L^{\dag})
(\varepsilon+\varepsilon^{*})
+2n\textrm{tr}(L^{\dag}L_{1}L^{\dag}L_{1}L^{\dag})
\varepsilon\varepsilon^{*}\\ \nonumber
&=n\textrm{tr}(L^{\dag})
-n\textrm{tr}((L^{\dag})^{2}L_{1})
(\varepsilon+\varepsilon^{*})
+2n\textrm{tr}((L^{\dag})^{2}L_{1}L^{\dag}L_{1})
\varepsilon\varepsilon^{*}.\end{align}
\end{proof}

\subsection{Quadratic forms of Hessian matrices for resistance distance and Kirchhoff index of  positive-weighted graphs and their
 estimations}
In this section, for a  connected  positive-weighted graph $G^{w}$, by using Eq. \eqref{jiq2} we
 establish relations between the quadratic forms of Hessian matrices (associated with its resistance distance and Kirchhoff index) and the Moore-Penrose inverse of its Laplacian
matrix. We further give estimates for the eigenvalues of these Hessian matrices in terms of graph parameters. Moreover, we prove that the Kirchhoff index
 of a  positive-weighted graph with bounded edge weights is strongly convex with respect to its edge weight vector..

\begin{thm}\label{ciio} Let $G^{w}$ be a connected  positive-weighted graph with $n$ vertices as in Eq. \eqref{chaoduiou}, whose edge weight vector is $\mathbf{x}$. Let $L$ be the Laplacian matrix of $G^{w}$ and $L_{1}$ the real matrix as in Eq. \eqref{lapu1}.
 Then the following statements hold.

$(a)$ The quadratic form for the Hessian matrix  of $R_{ij}(\mathbf{x})$ is $$(\Delta \mathbf{x})^{\top} \nabla^{2} R_{ij}(\mathbf{x})\Delta \mathbf{\mathbf{x}}=2(L^{\dag}L_{1}L^{\dag}L_{1}L^{\dag})^{(i,j)},$$
where $i,j\in[n]$.

$(b)$  The quadratic form for the Hessian matrix  of $Kf_{G}(\mathbf{x})$ is
$$(\Delta \mathbf{x})^{\top} \nabla^{2} Kf_{G}(\mathbf{x})\Delta \mathbf{x}= 2n \mathrm{tr}((L^\dag)^{2} L_{1} L^\dag L_{1}).$$
\end{thm}
\begin{proof}
 $(a)$ By Theorem \ref{5.1} $(b)$, we have
\begin{align}\nonumber R_{ij}(\widetilde{\mathbf{x}})&=R_{ij}(\textbf{x} +\Delta \textbf{x}(\varepsilon+\varepsilon^{*}))\\ \nonumber
&=(L^{\dag})^{(i,j)}-
(L^{\dag}L_{1}L^{\dag})^{(i,j)}(\varepsilon
+\varepsilon^{*})
+2(L^{\dag}L_{1}L^{\dag}L_{1}L^{\dag})^{(i,j)}
\varepsilon\varepsilon^{*}.
\end{align}
From Eq. \eqref{jiq2},  the quadratic form for the Hessian matrix  of $R_{ij}(\mathbf{x})$ is
\begin{align}\nonumber
(\Delta \textbf{x})^{\top}\nabla^{2}R_{ij}(\mathbf{x})\Delta \textbf{x}=\mathcal{S}_{\varepsilon\varepsilon^{*}}[R_{ij}(\textbf{x} +\Delta \textbf{x}(\varepsilon+\varepsilon^{*}))]=2(L^{\dag}L_{1}L^{\dag}L_{1}L^{\dag})^{(i,j)}.\end{align}

$(b)$ By Theorem \ref{5.1} $(c)$, we obtain
\begin{align*}
Kf_{G}(\widetilde{\textbf{x}})=& Kf_{G}(\mathbf{x} + \Delta \mathbf{x}(\varepsilon + \varepsilon^{*}))\\
=&n\textrm{tr}(L^{\dag})
-n\textrm{tr}((L^{\dag})^{2}L_{1})
(\varepsilon+\varepsilon^{*})
+2n\textrm{tr}((L^{\dag})^{2}L_{1}L^{\dag}L_{1})
\varepsilon\varepsilon^{*}.
\end{align*}
Again using Eq. \eqref{jiq2},
the quadratic form for the Hessian matrix  of $Kf_{G}(\textbf{x})$ is
\begin{align}\nonumber
(\Delta \textbf{x})^{\top}\nabla^2 Kf_{G}(\textbf{x})\Delta \textbf{x}=\mathcal{S}_{\varepsilon\varepsilon^{*}}[Kf(\textbf{x} +\Delta \textbf{x}(\varepsilon+\varepsilon^{*}))]= 2n \textrm{tr}((L^\dag)^{2} L_{1} L^\dag L_{1}).
\end{align}

\end{proof}
\begin{rmk}Ref. \cite{ghosh2008minimizing} used the matrix inverse to give the Hessian quadratic form of the Kirchhoff index $Kf(G^{w})$  of a connected  positive-weighted graph $G^{w}$. Since the Moore-Penrose inverse of the Laplacian matrix $L$ of $G^{w}$ is given by $L^{\dag}=(L+\frac{1}{n}J)^{-1}-\frac{1}{n}J$ \cite{bapat2004resistance},
 substituting this identity  into Theorem \ref{ciio}
$(b)$ in this paper directly leads to the result in \cite{ghosh2008minimizing}.
\end{rmk}
The second smallest eigenvalue of the Laplacian
matrix of a graph is called the algebraic
connectivity of this graph, which
has wide applications in investigating graph connectivity
and partitioning  \cite{chung1997spectral}. A graph is connected if and only if its algebraic connectivity is positive \cite{fiedler1973algebraic}. The largest eigenvalue of the Laplacian matrix plays an important role in the maximum cut problem \cite{mohar1990eigenvalues}.
Let $A\in \mathbb{R}^{n\times n}$
be positive semidefinite with
eigenvalues $\lambda_{1},\lambda_{2},\cdots,\lambda_{n}$. Then there exists an orthogonal matrix $P\in \mathbb{R}^{n\times n}$ such that $A=P\textrm{diag}(\lambda_{1},\cdots,\lambda_{n})P^{\top}$. Denote $A^{\frac{1}{2}}=P\textrm{diag}(\sqrt{\lambda_{1}},\cdots,
\sqrt{\lambda_{n}})P^{\top}.$

Let $L$ be the Laplacian matrix of a connected  positive-weighted graph $G^{w}$. The quantity $((L^{2})^{\dag})^{(i,j)}$ is the biharmonic distance of $G^{w}$\cite{lipman2010biharmonic}, denoted by $\widehat{R}_{ij}$. Subsequently, we give explicit  bounds on the eigenvalues of Hessian matrices for the resistance distance and Kirchhoff index of $G^{w}$, where the bounds are
 expressed in terms of the largest eigenvalue of the Laplacian matrix, algebraic connectivity, biharmonic distance of $G^{w}$, and the maximum degree of the underlying graph $G$.

\begin{thm}\label{thm4}
Let $G^{w}$ be a connected  positive-weighted graph with $n$ vertices as in \eqref{chaoduiou}, whose edge weight vector is $\mathbf{x}$. Let $\lambda_{1}$ and $\lambda_{n-1}$ be the largest eigenvalue of the Laplacian matrix $L$ and  algebraic connectivity  of $G^{w}$, respectively. Let $d_{max}$ be the maximum degree of the graph $G$. Then

$(a)$ The eigenvalue $\mu(\nabla^{2}R_{ij}(\mathbf{x}))$ of the Hessian matrix $\nabla^{2}R_{ij}(\mathbf{x})$ $($associated with  the resistance distance of $G^{w}$$)$ satisfies
\begin{align}
\nonumber \mu(\nabla^{2}R_{ij}(\mathbf{x}))\leq \frac{4(d_{max}+1)}{\lambda_{n-1}}\widehat{R}_{ij},\end{align}
where $\widehat{R}_{ij}$ is the biharmonic distance of $G^{w}$, $i,j\in[n].$

$(b)$  The eigenvalue $\mu(\nabla^{2}Kf_{G}(\mathbf{x}))$ of the Hessian matrix $\nabla^{2}Kf_{G}(\mathbf{x})$ $($associated with the Kirchhoff index  of $G^{w}$$)$ satisfies
\begin{align}
\nonumber \frac{4n}{\lambda_{1}^{3}}\leq \mu(\nabla^{2}Kf_{G}(\mathbf{x}))\leq \frac{4n(d_{max}+1)}{\lambda_{n-1}^{3}}.\end{align}
\end{thm}
\begin{proof}
 We first prove the following inequality\begin{align}\label{promi}2\|\Delta \textbf{x}\|^{2}_{2}\leq\|L_{1}\|_{F}^{2}\leq2(d_{max}+1)\|\Delta \textbf{x}\|^{2}_{2},\end{align}
where $\|\Delta \textbf{x}\|_{2}$ is the Euclidean norm  of the vector $\Delta \textbf{x}$, and $\|L_{1}\|_{F}$ is the Frobenius norm of the matrix $L_{1}$.
From Eq. \eqref{lapu1}, we obtain\begin{equation}\label{xuy1}
    \begin{aligned}
\|L_{1}\|_{F}^{2}=\textrm{tr}(L^{2}_{1})=\textrm{tr}((D_{1}-A_{1})^{2})
=\textrm{tr}(D_{1}^{2})+
\textrm{tr}(A_{1}^{2})-2\textrm{tr}(A_{1}D_{1}).
\end{aligned}
\end{equation}
Since all the diagonal elements of $A_{1}$ are  0 and $D_{1}$ is a diagonal matrix, it follows that $\textrm{tr}(A_{1}D_{1})=0.$ From Eq. \eqref{xuy1}, we have
\begin{align}\label{xu1}
\|L_{1}\|_{F}^{2}=\textrm{tr}(D_{1}^{2})+
\textrm{tr}(A_{1}^{2}).
\end{align}
By using Eqs. \eqref{chaoduiou} and \eqref{D1L}, it follows that \begin{align}\label{dd11}
 \textrm{tr}(A_{1}^{2})&=\sum_{i,j=1}^{n}(\Delta a_{ij})^{2}
=2\sum_{\{i,j\}\in E}\left(\Delta a_{ij}\right)^{2}=2\|\Delta \textbf{x}\|^{2}_{2},\\ \nonumber
\textrm{tr}(D_{1}^{2})&=\sum_{i}^{n} \left(\Delta d_{i}\right)^{2}= \sum_{i}^{n}\left(\sum_{i\sim j}^{n} \Delta a_{ij}\right)^{2}
.\end{align}
For a positive integer $s$  and real numbers $a_{1},a_{2},\cdots,a_{s}$, it is obvious that $(a_{1}+a_{2}+\cdots+a_{s})^{2}\leq s(a_{1}^{2}+a_{2}^{2}+\cdots+a_{s}^{2})$. Hence,
\begin{align}\nonumber \textrm{tr}(D_{1}^{2})&= \sum_{i}^{n}\left(\sum_{i\sim j}^{n} \Delta a_{ij}\right)^{2}\leq
\sum_{i}^{n}d_{i}\sum_{j=1}^{n} |\Delta a_{ij}|^{2}\\ \label{K2}
&\leq d_{max}\sum_{i,j=1}^{n} \Delta a_{ij}^{2}=2d_{max}\|\Delta \textbf{x}\|^{2}_{2},
\end{align}
where $d_{i}$ is the degree of the vertex $i$ of $G$, $i\in[n].$
Combining Eqs. \eqref{xu1}, \eqref{dd11} and \eqref{K2} gives $$\|L_{1}\|_{F}^{2}\leq2(d_{max}+1)\|\Delta \textbf{x}\|^{2}_{2}.$$ From \eqref{xu1}, it is obvious that $\|L_{1}\|_{F}^{2}\geq \textrm{tr}(A_{1}^{2})=2\|\Delta \textbf{x}\|^{2}_{2}.$ Thus, Ineq. \eqref{promi} is proved.

$(a)$
Let $\textbf{b}=\textbf{1}_{i}-\textbf{1}_{j}$.
 It follows from Theorem~\ref{ciio} $(a)$ that
\begin{align}\nonumber (\Delta \textbf{x})^{\top}\nabla^{2}R_{ij}(\mathbf{x})\Delta \textbf{x}&=2(L^{\dag}L_{1}L^{\dag}L_{1}L^{\dag})^{(i,j)}
=2\textbf{b}^{\top}L^{\dag}L_{1}L^{\dag}L_{1}L^{\dag}\textbf{b}\\ \nonumber
&=2\|(L^{\dag})^{\frac{1}{2}}L_{1}L^{\dag}\textbf{b}\|^{2}_{2}
\leq2\|(L^{\dag})^{\frac{1}{2}}\|_{2}^{2} \|L_{1}\|_{2}^{2} \|L^{\dag}\textbf{b}\|_{2}^{2}\\ \nonumber
&\leq 2\|(L^{\dag})^{\frac{1}{2}}\|_{2}^{2} \|L_{1}\|_{F}^{2} \|L^{\dag}\textbf{b}\|_{2}^{2},
\end{align} where $\|(L^{\dag})^{\frac{1}{2}}\|_{2}$ is the spectral norm of $(L^{\dag})^{\frac{1}{2}}$.

From the spectral decomposition Eq. \eqref{pfj1} of $L^{\dag}$, we obtain  $\|(L^{\dag})^{\frac{1}{2}}\|_{2}^{2}=\frac{1}{\lambda_{n-1}}.$ Then
\begin{align}\nonumber
(\Delta \textbf{x})^{\top}\nabla^{2}R_{ij}(\mathbf{x})\Delta \textbf{x}\leq 2\frac{ \|L_{1}\|_{F}^{2} \|L^{\dag}\textbf{b}\|_{2}^{2}}{\lambda_{n-1}}.
\end{align}
Again by  Ineq.  \eqref{promi}, we have \begin{align}\label{end1} \frac{(\Delta \textbf{x})^{\top}\nabla^{2}R_{ij}(\mathbf{x})\Delta \textbf{x}}{\|\Delta \textbf{x}\|^{2}}\leq\frac{4(d_{max}+1)}{\lambda_{n-1}}
\|L^{\dag}\textbf{b}\|_{2}^{2}=\frac{4(d_{max}+1)}{\lambda_{n-1}}((L^{2})^{\dag})^{(i,j)}.\end{align}
Therefore, $$\mu(\nabla^{2}R_{ij}(\mathbf{x}))\leq \frac{4(d_{max}+1)}{\lambda_{n-1}}((L^{2})^{\dag})^{(i,j)}=\frac{4(d_{max}+1)}{\lambda_{n-1}}\widehat{R}_{ij}.$$

$(b)$ For $A\in \mathbb{R}^{m\times n}$ and $B\in \mathbb{R}^{n\times p}$, the norm inequality \cite[p.~364]{horn2012matrix1}\begin{align}\label{bds}\|AB\|_{F}\leq \|A\|_{2}\|B\|_{F}.\end{align} holds.
From Theorem~\ref{ciio} $(b)$ and  Ineq.  \eqref{bds}, we have\begin{align}\label{K4} (\Delta \mathbf{x})^{\top} \nabla^{2} Kf_{G}(\Delta \mathbf{x}) \Delta \mathbf{\mathbf{x}}&=2n\textrm{tr} ( (L^\dag)^{2} L_{1} L^\dag L_{1})=2n\|(L^{\dag})^{\frac{1}{2}}L_{1}L^{\dag}\|^{2}_{F}\\ \nonumber
&\leq 2n\|(L^{\dag})^{\frac{1}{2}}\|_{2}^{2}\|L_{1}L^{\dag}\|_{F}^{2}
\leq 2n\|(L^{\dag})^{\frac{1}{2}}\|_{2}^{2}\|L_{1}\|_{F}^{2}\|L^{\dag}\|_{2}^{2}.
\end{align}
By Eq. \eqref{pfj1}, we have $\|(L^{\dag})^{\frac{1}{2}}\|_{2}^{2}=\frac{1}{\lambda_{n-1}}$
and $\|L^{\dag}\|_{2}^{2}=\frac{1}{\lambda_{n-1}^{2}}.$
Then \begin{align}\nonumber
(\Delta \mathbf{x})^{\top} \nabla^{2} Kf_{G}(\Delta \mathbf{x}) \Delta \mathbf{\mathbf{x}}\leq2n\frac{\|L_{1}\|_{F}^{2}}
{\lambda_{n-1}^{3}}.
\end{align}
Again by  Ineq.  \eqref{promi}, we get $$\frac{(\Delta \textbf{x})^{\top} \nabla^{2} Kf_{G}(\textbf{x})\Delta \textbf{x}}{\|\Delta \textbf{x} \|^{2}_{2}}\leq\frac{4n(d_{max}+1)}{\lambda_{n-1}^{3}}.$$
 Thus, $$\mu(\nabla^{2}Kf_{G}(\mathbf{x}))\leq \frac{4n(d_{max}+1)}{\lambda_{n-1}^3}.$$

Next we derive a lower bound on $\mu(\nabla^{2}Kf_{G}(\mathbf{x}))$. Applying  Ineq.  \eqref{bds} gives
$\|L^{\frac{1}{2}}(L^{\dag})^{\frac{1}{2}}L_{1}L^{\dag}\|_{F}^{2}\leq \|L^{\frac{1}{2}}\|_{2}^{2}
\|(L^{\dag})^{\frac{1}{2}}L_{1}L^{\dag}\|^{2}_{F}.$
Then
$$\|(L^{\dag})^{\frac{1}{2}}L_{1}L^{\dag}\|^{2}_{F}\geq
\frac{\|L^{\frac{1}{2}}(L^{\dag})^{\frac{1}{2}}L_{1}L^{\dag}\|_{F}^{2}}{\|L^{\frac{1}{2}}\|_{2}^{2}}.$$
 Combining this with Eq. \eqref{K4}, we have \begin{align}\nonumber
&(\Delta \mathbf{x})^{\top} \nabla^{2} Kf_{G}(\Delta \mathbf{x}) \Delta \mathbf{\mathbf{x}}
=2n\|(L^{\dag})^{\frac{1}{2}}L_{1}L^{\dag}\|^{2}_{F}\geq
2n\frac{\|L^{\frac{1}{2}}(L^{\dag})^{\frac{1}{2}}L_{1}L^{\dag}\|_{F}^{2}}{\|L^{\frac{1}{2}}\|_{2}^{2}}.
\end{align}
From Eqs. \eqref{pfj} and \eqref{pfj1}, we have $L^{\frac{1}{2}}(L^{\dag})^{\frac{1}{2}}=(\sum_{i=1}^{n-1}\sqrt{\lambda_{i}}p_{i}p_{i}^{\top})
(\sum_{i=1}^{n-1}\frac{1}{\sqrt{\lambda_{i}}}p_{i}p_{i}^{\top})=\sum_{i=1}^{n-1}p_{i}p_{i}^{\top}=p_{n}p_{n}^{\top}
=I-\frac{1}{n}J.$ It follows that
\begin{align}\label{kfg3}(\Delta \mathbf{x})^{\top} \nabla^{2} Kf_{G}(\Delta \mathbf{x}) \Delta \mathbf{\mathbf{x}}
\geq 2n\frac{\|(I-\frac{1}{n}J)L_{1}L^{\dag}\|_{F}^{2}}{\|L^{\frac{1}{2}}\|_{2}^{2}}
=2n\frac{\|L_{1}L^{\dag}\|_{F}^{2}}{\|L^{\frac{1}{2}}\|_{2}^{2}}.\end{align}
Again by  Ineq.  \eqref{bds}, we get
$\|L_{1}L^{\dag}L\|_{F}^{2}\leq \|L_{1}L^{\dag}\|_{F}^{2}\|L\|_{2}^{2}.$ Then
\begin{align}\label{zuo1}\|L_{1}L^{\dag}\|_{F}^{2}\geq \frac{\|L_{1}L^{\dag}L\|_{F}^{2}}{\|L\|_{2}^{2}}.\end{align}
Combining  Ineqs.  \eqref{kfg3} with \eqref{zuo1}, we obtain
$$(\Delta \mathbf{x})^{\top} \nabla^{2} Kf_{G}(\Delta \mathbf{x}) \Delta \mathbf{\mathbf{x}}
\geq2n\frac{\|L_{1}L^{\dag}L\|_{F}^{2}}{\|L^{\frac{1}{2}}\|_{2}^{2}\|L\|_{2}^{2}}
=2n\frac{\|L_{1}(I-\frac{1}{n}J)\|_{F}^{2}}{\lambda_{n-1}^{3}}
=2n\frac{\|L_{1}\|_{F}^{2}}{\lambda_{n-1}^{3}}.$$
Again by  Ineq.  \eqref{promi}, we have $$\frac{(\Delta \textbf{x})^{\top} \nabla^{2} Kf_{G}(\textbf{x})\Delta \textbf{x}}{\|\Delta \textbf{x} \|_{2}^{2}}\geq\frac{4n}{\lambda_{1}^{3}},$$
then $$\mu(\nabla^{2}Kf_{G}(\mathbf{x}))\geq \frac{4n}{\lambda_{1}^{3}}.$$
This completes the proof of part $(b)$ of Theorem~\ref{thm4}.

\end{proof}
Since  $\|L^{\dag}\textbf{b}\|_{2}^{2}\leq \|L^{\dag}\|_{2}^{2}\|\|\textbf{b}\|_{2}^{2}
  =\frac{2}{\lambda_{n-1}^{2}}$  in Ineq. \eqref{end1},  we obtain another upper bound for Theorem \ref{thm4} $(a)$.
\begin{cor}\label{cor1}
  For a connected  positive-weighted graph $G^{w}$ with $n$ vertices, the eigenvalue $\mu(\nabla^{2}R_{ij}(\mathbf{x}))$ of the Hessian matrix $\nabla^{2}R_{ij}(\mathbf{x})$ $($associated with  the resistance distance of $G^{w}$$)$ satisfies
$$\mu(\nabla^{2}R_{ij}(\mathbf{x}))\leq \frac{8(d_{max}+1)}{\lambda_{n-1}^3},$$
where  $i,j\in[n]$.
\end{cor}
 Theorem \ref{thm4} and Corollary \ref{cor1} give  bounds on the eigenvalues of the Hessian matrix of the resistance distance and the Kirchhoff index of a  positive-weighted graph in terms of graph parameters. Two examples are given below to verify these bounds.

\begin{example}\label{rem3.6}
$(1)$ Let $K_{3}$ be a complete graph with the vertex set $V=\{1,2,3\}$ and edge set $E=\{e_{1}=\{1,2\},e_{2}=\{2,3\},e_{3}=\{1,3\}\}$.
Let $K_{3}^{w}=(V,E,w)$ be a positive-weighted graph with $K_{3}$ as its underlying graph, where  $w(e_{i})=x_{i}$ is the weight of edge $e_{i}$, $i=1,2,3.$ By using the series and parallel manipulations, the resistance distance between vertices $1$ and $2$ of
$K_{3}^{w}$ is
$$R_{12}(K_{3}^{w})=
\frac{x_{2}+x_{3}}{x_{1}x_{2}+x_{2}x_{3}+x_{1}x_{3}}=R_{12}(\mathbf{x}).$$
The second order partial derivatives of $R_{12}(\mathbf{x})$ are:
\begin{align}\nonumber \frac{\partial^{2} R_{12}(\mathbf{x})}{\partial x_{1}\partial x_{2}}&=
\frac{2x_{3}^{2}(x_{2}+x_{3})}
{D^{3}},\ \ \frac{\partial^{2} R_{12}(\mathbf{x})}{\partial x_{1}\partial x_{3}}=\frac{2x_{2}^{2}(x_{2}+x_{3})}
{D^{3}},\ \
\frac{\partial^{2} R_{12}(\mathbf{x})}{\partial x_{2}\partial x_{3}}=-\frac{2x_{1}x_{2}x_{3}}
{D^{3}},\\ \nonumber
\frac{\partial^{2} R_{12}(\mathbf{x})}{\partial x_{1}^{2}}&=
\frac{2(x_{2}+x_{3})^{3}}
{D^{3}},\ \ \frac{\partial^{2} R_{12}(\mathbf{x})}{\partial x_{2}^{2}}=
\frac{2x_{3}^{2}(x_{1}+x_{2})}
{D^{3}},\ \ \frac{\partial^{2} R_{12}(\mathbf{x})}{\partial x_{3}^{2}}=
\frac{2x_{2}^{2}(x_{1}+x_{2})}
{D^{3}},
\end{align}
where $D=x_{1}x_{2}+x_{2}x_{3}+x_{1}x_{3}$. Setting $x_{1}=x_{2}=x_{3}=1$, the Hessian matrix
$\nabla^{2}R_{12}(\mathbf{x})$ becomes
$$\nabla^{2}R_{12}(\mathbf{1})=\left[\begin{array}{ccc}
0.5926  &  0.1481  &  0.1481\\
    0.1481 &   0.1481 &  -0.0741\\
    0.1481 &  -0.0741  &  0.1481
    \end{array}
\right].$$
The largest eigenvalue of the above Hessian matrix is
0.6667.
Similarly, the largest eigenvalues of $\nabla^{2}R_{23}(\mathbf{1})$ and
$\nabla^{2}R_{13}(\mathbf{1})$ are both
0.6667.

It is easy to see that the largest eigenvalue and the algebraic connectivity of the Laplacian matrix of $K_{3}$ are both $3$. From Corollary \ref{cor1}, the upper bound represented by graph parameters for the eigenvalues of the Hessian matrix of the resistance distance of $K_3$ is calculated as  $0.8888$.

$(2)$ Similar to $(1)$ above, the Hessian matrix of the Kirchhoff index of $K_{4}$ is
$$ \nabla^{2}Kf_{K_{4}}(\textbf{1})=\left[\begin{array}{cccccc}
    0.5000 &   0.1250  &  0.1250  &  0.1250 &   0.1250 &   0.0000\\
    0.1250  &  0.5000  &  0.1250 &   0.1250  &  0.0000 &   0.1250\\
    0.1250  &  0.1250 &   0.5000   & 0.0000 &   0.1250  &  0.1250\\
    0.1250  &  0.1250   & 0.0000  &  0.5000 &   0.1250 &   0.1250\\
    0.1250  &  0.0000  &  0.1250  &  0.1250 &   0.5000 &   0.1250\\
    0.0000  &  0.1250   & 0.1250  &  0.1250 &   0.1250 & 0.5000\end{array}
\right].$$
The smallest and largest eigenvalues of $\nabla^{2}Kf_{K_4}(\mathbf{1})$ are computed to be $\frac{1}{4}$ and 1, respectively.

It is easy to see that the largest eigenvalue and the algebraic connectivity of the Laplacian matrix of the complete graph $K_{4}$ are both $4$. From Theorem \ref{thm4} $(b)$, the lower bound and upper bound for the eigenvalues of the Hessian matrix of the Kirchhoff index of $K_{4}$ are $\frac{1}{4}$ and $1$, respectively.

\end{example}

It is well known that a function $f(\textbf{x})$ is convex if and only if its Hessian matrix is positive semidefinite. If the Hessian matrix of
$f(\textbf{x})$ is positive definite, then $f(\textbf{x})$ is strictly convex. Moreover, if there exists a constant $\alpha>0$ such that the smallest eigenvalue of the Hessian matrix of  $f(\textbf{x})$ is greater than or equal to $\alpha$, then  $f(\textbf{x})$ is strongly convex \cite{boyd2004convex1}.

Ghosh, Boyd and Saberi proved that the resistance distance and Kirchhoff index of a connected  positive-weighted graph $G^{w}$ are strictly convex functions on its  edge weight vector $\mathbf{x}$ \cite{ghosh2008minimizing}. In the  following, we show that for a  positive-weighted graph $G^{w}$ with bounded edge weights, the Kirchhoff index  $Kf_{G}(\textbf{x})$  is strongly convex on its edge weight $\textbf{x}$.
\begin{thm}
The Kirchhoff index of a connected  positive-weighted graph with $n$ fixed vertices and bounded edge weights is strongly convex with respect to its edge weight vector.
\end{thm}
\begin{proof}
It follows from \cite{anderson1985eigenvalues} that the largest eigenvalue $\lambda_{1}$  of the Laplacian matrix of a  positive-weighted graph $G^{w}$ satisfies
$\lambda_{1}\leq 2 d_{max}^{w}$, where $d_{max}^{w}$ is the maximum weighted degree of $G^{w}$.  Let  $w(e)\leq M$ for all $e\in E$, where $M>0$ is  a constant.  Then $$\lambda_{1}\leq 2d_{max}^{w}\leq 2(n-1)M.$$
Let $\alpha=\frac{n}{2(n-1)^{3}M^{3}}.$ By Theorem~\ref{thm4} $(b)$, the smallest eigenvalue $\mu_{min}(\nabla^{2}Kf_{G}(\textbf{x}))$ $($associated with the Hessian matrix of the Kirchhoff index of $G^{w}$$)$ satisfies
$$\mu_{min}(\nabla^{2}Kf_{G}(\textbf{x}))\geq \frac{4n}{\lambda_{1}^{3}}\geq \frac{n}{2(n-1)^{3}M^{3}}=\alpha
>0.$$  Hence, the Kirchhoff index of $G^{w}$  is strongly convex with respect to its edge weight vector $\textbf{x}$.
\end{proof}
\section*{Acknowledgement}
The research of the second author is partially
supported by the National Natural Science Foundation
 of China (No.12371344), the Natural Science Foundation
  for The Excellent Youth Scholars of the Heilongjiang Province (No.YQ2024A009) and the Fundamental Research Funds
  for the Central Universities (No.3072025YC2403),
  and the third author is partially supported by the
  National Natural Science Foundation of China
  (No.12071097).

\section*{References}
\bibliographystyle{unsrt}
\bibliography{scholar}

\begin{thebibliography}{10}

\bibitem{harary1969graph}
F.~Harary.
\newblock {\em Graph {T}heory}.
\newblock Addison-Wesley Publishing company, United States, 1969.

\bibitem{foster1949average}
R.M. Foster.
\newblock The average impedance of an electrical network.
\newblock {\em Contributions to Applied Mechanics {(Reissner Anniversary
  Volume)}}, pages 333--340, 1949.

\bibitem{chen2008resistance}
H.~Chen and F.~Zhang.
\newblock Resistance distance local rules.
\newblock {\em Journal of Mathematical Chemistry}, 44(2):405--417, 2008.

\bibitem{Klein1993Resistance}
D.J. Klein and M.~Randi\'{c}.
\newblock Resistance distance.
\newblock {\em Journal of Mathematical Chemistry}, 12:81--95, 1993.

\bibitem{bapat2004resistance}
R.B. Bapat.
\newblock Resistance matrix of a weighted graph.
\newblock {\em MATCH Communications in Mathematical and in Computer Chemistry},
  50:73--82, 2004.

\bibitem{cheng2022counting}
S.~Cheng, W.~Chen, and W.~Yan.
\newblock Counting spanning trees in almost complete multipartite graphs.
\newblock {\em Journal of Algebraic Combinatorics}, 56:773--783, 2022.

\bibitem{li2019enumeration}
T.~Li and W.~Yan.
\newblock Enumeration of spanning trees of 2-separable networks.
\newblock {\em Physica A: Statistical Mechanics and its Applications},
  536:120877, 2019.

\bibitem{lovasz1993random}
L.~Lov{\'a}sz.
\newblock Random walks on graphs: A survey.
\newblock {\em Combinatorics, Paul erdos is eighty}, 2:1--46, 1993.

\bibitem{cloninger2024random}
A.~Cloninger, G.~Mishne, A.~Oslandsbotn, S.J. Robertson, Z.~Wan, and Y.~Wang.
\newblock {Random walks, conductance, and resistance for the connection graph
  Laplacian}.
\newblock {\em SIAM Journal on Matrix Analysis and Applications},
  45(3):1541--1572, 2024.

\bibitem{zhang2019detecting}
T.~Zhang and C.~Bu.
\newblock Detecting community structure in complex networks via resistance
  distance.
\newblock {\em Physica A: Statistical Mechanics and its Applications},
  526:120782, 2019.

\bibitem{bonchev1994molecular}
D.~Bonchev, A.T. Balaban, X.~Liu, and D.J. Klein.
\newblock Molecular cyclicity and centricity of polycyclic graphs. {I}.
  {C}yclicity based on resistance distances or reciprocal distances.
\newblock {\em International journal of quantum chemistry}, 50(1):1--20, 1994.

\bibitem{yang2018designing}
C.~Yang, J.~Mao, X.~Qian, and P.~Wei.
\newblock Designing robust air transportation networks via minimizing total
  effective resistance.
\newblock {\em IEEE Transactions on Intelligent Transportation Systems},
  20(6):2353--2366, 2018.

\bibitem{ellens2011effective}
W.~Ellens, F.~Spieksma, P.~Van~Mieghem, A.~Jamakovic, and R.~Kooij.
\newblock Effective graph resistance.
\newblock {\em Linear algebra and its applications}, 435(10):2491--2506, 2011.

\bibitem{shannon1993concavity}
C.~Shannon and D~Hagelberger.
\newblock Concavity of resistance functions.
\newblock {\em Journal of Applied Physics}, 27:42--43, 1956.

\bibitem{ghosh2008minimizing}
A.~Ghosh, S.~Boyd, and A.~Saberi.
\newblock Minimizing effective resistance of a graph.
\newblock {\em SIAM review}, 50(1):37--66, 2008.

\bibitem{li2025resistance}
Y.~Li, L.~Sun, and C.~Bu.
\newblock The resistance distance of a dual number weighted graph.
\newblock {\em Discrete Applied Mathematics}, 375:154--165, 2025.

\bibitem{chen2021fast}
C.~Chen, S.~Reiz, C.~Yu, H.~Bungartz, and G.~Biros.
\newblock Fast approximation of the {G}auss--{N}ewton {H}essian matrix for the
  multilayer perceptron.
\newblock {\em SIAM Journal on Matrix Analysis and Applications},
  42(1):165--184, 2021.

\bibitem{yazawa2021eigenvalues}
A.~Yazawa.
\newblock The eigenvalues of {H}essian matrices of the complete and complete
  bipartite graphs.
\newblock {\em Journal of Algebraic Combinatorics}, 54(4):1137--1157, 2021.

\bibitem{pan2019principal}
Y.~Pan, Y.~Zhou, W.~Liu, and L.~Nie.
\newblock Principal component analysis on graph-{H}essian.
\newblock In {\em 2019 IEEE Symposium Series on Computational Intelligence
  (SSCI)}, pages 1494--1501. IEEE, 2019.

\bibitem{clifford1871preliminary}
M.A. Clifford.
\newblock Preliminary sketch of biquaternions.
\newblock {\em Proceedings of the London Mathematical Society}, 4(64):381--395,
  1873.

\bibitem{Geometry1010}
E.~Study.
\newblock {\em Geometry der Dynamen}.
\newblock Verlag Teubner, Leipzig, 1903.

\bibitem{fischer}
I.~Fischer.
\newblock {\em Dual-Number Methods in Kinematics, Statics and Dynamics}.
\newblock CRC Press, Boca Raton, 1998.

\bibitem{gu1987dual}
Y.~Gu and J.~Luh.
\newblock Dual-number transformation and its applications to robotics.
\newblock {\em IEEE Journal on Robotics and Automation}, 3(6):615--623, 1987.

\bibitem{qi2024eigenvalues}
L.~Qi and C.~Cui.
\newblock Eigenvalues of dual {H}ermitian matrices with application in
  formation control.
\newblock {\em SIAM Journal on Matrix Analysis and Applications},
  45(4):2135--2154, 2024.

\bibitem{wei2024singular}
T.~Wei, W.~Ding, and Y.~Wei.
\newblock Singular value decomposition of dual matrices and its application to
  traveling wave identification in the brain.
\newblock {\em SIAM Journal on Matrix Analysis and Applications},
  45(1):634--660, 2024.

\bibitem{qi2024dualww}
L.~Qi and C.~Cui.
\newblock Dual markov chain and dual number matrices with nonnegative standard
  parts.
\newblock {\em Communications on Applied Mathematics and Computation}, pages
  1--20, 2024.

\bibitem{fike2011development}
J.~Fike and J.~Alonso.
\newblock The development of hyper-dual numbers for exact second-derivative
  calculations.
\newblock In {\em 49th AIAA aerospace sciences meeting including the new
  horizons forum and aerospace exposition}, page 886, 2011.

\bibitem{fike2012automatic}
J.~Fike and J.~Alonso.
\newblock Automatic differentiation through the use of hyper-dual numbers for
  second derivatives.
\newblock In {\em Recent Advances in Algorithmic Differentiation}, pages
  163--173, 2012.

\bibitem{murai2022method}
D.~Murai, R.~Omote, and M.~Tanaka.
\newblock The method for solving topology optimization problems using
  hyper-dual numbers.
\newblock {\em Archive of Applied Mechanics}, 92(10):2813--2824, 2022.

\bibitem{cohen2016application}
A.~Cohen and M.~Shoham.
\newblock Application of hyper-dual numbers to multibody kinematics.
\newblock {\em Journal of Mechanisms and Robotics}, 8(1):011015, 2016.

\bibitem{tanaka2015highly}
M.~Tanaka, T.~Sasagawa, R.~Omote, M.~Fujikawa, D.~Balzani, and J.~Schr{\"o}der.
\newblock A highly accurate 1st-and 2nd-order differentiation scheme for
  hyperelastic material models based on hyper-dual numbers.
\newblock {\em Computer Methods in Applied Mechanics and Engineering},
  283:22--45, 2015.

\bibitem{xiao2024c}
Q.~Xiao and J.~Zhong.
\newblock Characterizations and properties of hyper-dual
  \textrm{M}oore-\textrm{P}enrose generalized inverse.
\newblock {\em AIMS Mathematics}, 9(12):35125--35150, 2024.

\bibitem{bapat2010graphs}
R.B. Bapat.
\newblock {\em Graphs and Matrices}.
\newblock Springer, London, 2010.

\bibitem{chung1997spectral}
F.~Chung.
\newblock {\em Spectral {G}raph {T}heory}.
\newblock American Mathematical Society, Providence, 1997.

\bibitem{fiedler1973algebraic}
M.~Fiedler.
\newblock Algebraic connectivity of graphs.
\newblock {\em Czechoslovak mathematical journal}, 23(2):298--305, 1973.

\bibitem{mohar1990eigenvalues}
B.~Mohar and S.~Poljak.
\newblock Eigenvalues and the max-cut problem.
\newblock {\em Czechoslovak Mathematical Journal}, 40(2):343--352, 1990.

\bibitem{lipman2010biharmonic}
Y.~Lipman, R.~Rustamov, and T.~Funkhouser.
\newblock Biharmonic distance.
\newblock {\em ACM Transactions on Graphics (TOG)}, 29(3):1--11, 2010.

\bibitem{horn2012matrix1}
R.~Horn and C.~Johnson.
\newblock {\em Matrix {A}nalysis}.
\newblock Cambridge university press, New York, 2013.

\bibitem{boyd2004convex1}
S.~Boyd and L.~Vandenberghe.
\newblock {\em Convex {O}ptimization}.
\newblock Cambridge university press, Cambridge, UK, 2004.

\bibitem{anderson1985eigenvalues}
W.N. Anderson and T.D. Morley.
\newblock Eigenvalues of the {L}aplacian of a graph.
\newblock {\em Linear and multilinear algebra}, 18(2):141--145, 1985.

\end{thebibliography}
\end{spacing}
\end{document}